\documentclass[oneside,reqno]{amsart}
\usepackage{amssymb}
\usepackage{amsmath}
\numberwithin{equation}{section}
\newcommand{\R}{\mathbb R}
\newcommand{\N}{\mathbb N}

\newtheorem{theorem}{Theorem}[section]

\newtheorem{lem}{Lemma}[section]
\newtheorem{rem}{Remark}[section]

\begin{document}
\title{Generalized Positive Linear Operators based on PED and IPED}

\author{Naokant Deo and Minakshi Dhamija}

\date{}

\address{Delhi Technological University\newline
\indent Formerly Delhi College of Engineering\newline
\indent Department of Applied Mathematics\newline
\indent Bawana Road, 110042 Delhi, India}
\email{naokantdeo@dce.ac.in, dr\_naokant\_deo@yahoo.com}
\email{minakshidhamija11@gmail.com}


\begin{abstract}
The paper deals with  generalized positive linear operators based on P\'olya-Eggenberger distribution(PED) as well as inverse P\'olya-Eggenberger distribution(IPED). Initially, we give the moments by using Stirling numbers of second kind and then establish direct results, convergence with the help of local and weighted approximation of proposed operators.
\end{abstract}

\maketitle\markboth{Naokant Deo and Minakshi Dhamija}{Generalized Positive Linear Operators}

\section{Introduction}

In the year $1968$, D. D. Stancu \cite{Stancu:68} introduced a new class of positive linear operators based on P\'{o}lya-Eggenberger distribution(PED) and associated to a real-valued function on $[0,1]$ as:
\begin{equation}\label{B1}
B_n^{(\alpha )}\left( {f;x} \right) =\sum\limits_{k = 0}^n \binom{n}{k}\frac{x^{[k,-\alpha]}(1-x)^{[n-k,-\alpha]}}{1^{[n,-\alpha]}}f\left({\frac{k}{n}}\right),
\end{equation}
where $\alpha$ is a non-negative parameter which may depend only on the natural number $n$ and $t^{[n,h]} = t(t - h)(t - 2h)\cdot\ldots\cdot(t - \overline{n - 1}h)$, $t^{[0,h]} = 1$ represents the factorial power of $t$ with increment $h$.

Later on Stancu \cite{Stancu:70} introduced a generalized form of the Baskakov operators based on inverse P\'{o}lya-Eggenberger distribution(IPED) for a real-valued function bounded on $\left[ {0,\infty } \right)$, given by
\begin{equation}\label{V1}
V_n^{(\alpha )}\left( {f;x} \right) =\sum\limits_{k=0}^{\infty}\binom{n+k-1}{k}\frac{1^{[n,-\alpha]}x^{[k,-\alpha]}}{(1+x)^{[n+k,-\alpha]}}f\left( {\frac{k}{n}}\right).
\end{equation}

Now we consider new positive linear operators ${L_n^{(\alpha )}}$, for each $f$, real valued function bounded on interval $I$, as:
\begin{equation}\label{L1}
 {L_n^{(\alpha )}(f;x)}= \sum\limits_{k} {w_{n,k}^{(\alpha )}(x)f\left( {\frac{k}{n}} \right)},\;\;\;x \in I,\;\;n=1,2,...,
\end{equation}
where $\alpha = \alpha(n) \rightarrow 0~~~ \text{when}~~~ n \rightarrow \infty$, $p \;\&\; k$ are nonnegative integers and for $\lambda =-1,0$, we have
\begin{align*}
 \omega _{n,k}^{(\alpha )}(x) &= \frac{{n + p}}{{n + p + \overline {\lambda  + 1} k}}\left( {\begin{array}{*{20}{c}}
  {n + p + \overline {\lambda  + 1} k} \\
  k
\end{array}} \right)\frac{{\prod\limits_{i = 0}^{k - 1} {\left( {x + i\alpha } \right)} \prod\limits_{i = 0}^{n + p + \lambda k - 1} {\left( {1 + \lambda x + i\alpha } \right)} }}{{\prod\limits_{i = 0}^{n + p + \overline {\lambda  + 1} k - 1} {\left( {1 + \overline {\lambda  + 1} x + i\alpha } \right)} }}\\
&= \frac{{n + p}}{{n + p + \overline {\lambda  + 1} k}}\left( {\begin{array}{*{20}{c}}
{n + p + \overline {\lambda  + 1} k}\\
k
\end{array}} \right)\frac{{{x^{[k, - \alpha ]}}{{(1 + \lambda x)}^{[n + p + \lambda k, - \alpha ]}}}}{{{{\left( {1 + \overline {\lambda  + 1} x} \right)}^{[n + p + \overline {\lambda  + 1} k, - \alpha ]}}}},
\end{align*}
by using the notation $ \overline {m-r}\alpha=(m-r)\alpha$. Operators \eqref{L1} is generalized form of above two operators \eqref{B1} and \eqref{V1} and associated with PED and IPED ~\cite{EP:23}.

Kantorovich form of Stancu operators~\eqref{B1} had been given by Razi~\cite{Razi:89} and he studied its convergence properties and degree of approximation. Ispir et al.~\cite{IAK:15} also discussed the Kantorovich form of operators~\eqref{B1} and they estimated the rate of convergence for absolutely continuous functions having a derivative coinciding a.e. with a function of bounded variation.
Recently, Miclaus \cite{Miclaus:14} established some approximation results for the Stancu operators~\eqref{B1} and its Durrmeyer type integral modification was studied by Gupta and Rassias \cite{GR:14} and obtained some direct results which include an asymptotic formula, local and global approximation results for these operators in terms of modulus of continuity.

Very recently Durrmeyer type modification of generalized Baskakov operators~\eqref{V1} associated with IPED were introduced by Dhamija and Deo~\cite{DD:16} and studied the moments with the help of Vandermonde convolution formula and then gave approximation properties of these operators which include uniform convergence and degree of approximation. Deo et al.~\cite{DDM:16} also investigated approximation properties of Kantorovich variant of operators~\eqref{V1} and they established uniform convergence, asymptotic formula and degree of approximation. Various Durrmeyer type modifications and then their local as well as weighted approximation along with some other approximation behaviour have been discussed by many authors e.g.,(~\cite{Deo:12},~\cite{JDD:14},~\cite{GA:14}).

The main object of this paper is to find moments with the help of Stirling numbers of second kind and estimate the rate of convergence of operators~\eqref{L1} by studing local and weighted approximation theorems.

Throughout this paper we consider interval $I=\left[ {0,\infty } \right)$ for $\lambda  = 0$ and $I=\left[ {0,1} \right]$ for $\lambda  =  - 1$.

\section{Special Cases}
It is easy to understand that the special cases of operators~(\ref{L1}) are as follows:
\begin{enumerate}
\item For $\lambda  =  - 1$; we have
\begin{enumerate}
\item [(i)] When $\alpha  \ne 0 \ne p$
\[{L_n^{(\alpha )}(f;x)} = \sum\limits_{k = 0}^{n + p} {\left( {\begin{array}{*{20}{c}}
  {n + p} \\
  k
\end{array}} \right)} \frac{{\prod\limits_{i = 0}^{k - 1} {\left( {x + i\alpha } \right)} \prod\limits_{i = 0}^{n + p - k - 1} {\left( {1 - x + i\alpha } \right)} }}{{\prod\limits_{i = 0}^{n + p - 1} {\left( {1 + i\alpha } \right)} }}f\left( {\frac{k}{n}} \right).\]
This leads to Schurer type Stancu operators.
\item [(ii)] When $\alpha  \ne 0,p = 0$ we get alternate form of operators~\eqref{B1} as:
\[{L_n^{(\alpha )}(f;x)} = \sum\limits_{k = 0}^n {\left( {\begin{array}{*{20}{c}}
  n \\
  k
\end{array}} \right)} \frac{{\prod\limits_{i = 0}^{k - 1} {\left( {x + i\alpha } \right)} \prod\limits_{i = 0}^{n - k - 1} {\left( {1 - x + i\alpha } \right)} }}{{\prod\limits_{i = 0}^{n - 1} {\left( {1 + i\alpha } \right)} }}f\left( {\frac{k}{n}} \right).\]
\textbf{\emph{Particular case:}} when $\alpha  = 1/n, p = 0$ we obtain Lupa\c{s} and Lupa\c{s} operators~\cite{Lapus:87} as:
   \[{L_n^{(\alpha )}(f;x)} = \frac{{2\left( {n!} \right)}}{{\left( {2n} \right)!}}\sum\limits_{k = 0}^n {\left( {\begin{array}{*{20}{c}}
  n \\
  k
\end{array}} \right)} \prod\limits_{i = 0}^{k - 1} {\left( {nx + i} \right)} \prod\limits_{i = 0}^{n - k - 1} {\left( {\overline {1 - x} n + i} \right)} f\left( {\frac{k}{n}} \right).\]
\item [(iii)] When $\alpha  = 0, p \ne 0$ we have Bernstein-Schurer operators~\cite{Schurer:62} as:
\[{L_n^{(\alpha )}(f;x)} = \sum\limits_{k = 0}^{n + p} {\left( {\begin{array}{*{20}{c}}
  {n + p} \\
  k
\end{array}} \right){x^k}{{(1 - x)}^{n + p - k}}f\left( {\frac{k}{n}} \right)}.\]
\item [(iv)]  When $\alpha  = 0, p = 0$ we obtain original Bernstein operators~\cite{Bernstein:12} as:
\[{L_n^{(\alpha )}(f;x)} = \sum\limits_{k = 0}^n {\left( {\begin{array}{*{20}{c}}
  n \\
  k
\end{array}} \right){x^k}{{(1 - x)}^{n - k}}f\left( {\frac{k}{n}} \right)} .\]
\end{enumerate}

\item For $\lambda  = 0$; we obtain the following operators:
\begin{enumerate}
\item [(i)] When $\alpha  \ne 0 \ne p$ we have Stancu-Schurer operators as:
\[{L_n^{(\alpha )}(f;x)} = \sum\limits_{k = 0}^\infty  {\left( {\begin{array}{*{20}{c}}
  {n + p + k - 1} \\
  k
\end{array}} \right)\frac{{\prod\limits_{i = 0}^{k - 1} {\left( {x + i\alpha } \right)} \prod\limits_{i = 0}^{n + p - 1} {\left( {1 + i\alpha } \right)} }}{{\prod\limits_{i = 0}^{n + p + k - 1} {\left( {1 + x + i\alpha } \right)} }}f\left( {\frac{k}{n}} \right),} \]
\item [(ii)] When $\alpha  \ne 0,p = 0$ we obtain alternate form of operators~\eqref{V1} as:
\[{L_n^{(\alpha )}(f;x)} = \sum\limits_{k = 0}^\infty  {\left( {\begin{array}{*{20}{c}}
  {n + k - 1} \\
  k
\end{array}} \right)\frac{{\prod\limits_{i = 0}^{k - 1} {\left( {x + i\alpha } \right)} \prod\limits_{i = 0}^{n - 1} {\left( {1 + i\alpha } \right)} }}{{\prod\limits_{i = 0}^{n + k - 1} {\left( {1 + x + i\alpha } \right)} }}f\left( {\frac{k}{n}} \right),} \]
\item [(iii)] When $\alpha  = 0,p \ne 0$ we get Baskakov-Schurer operators as:
\[{L_n^{(\alpha )}(f;x)} = \sum\limits_{k = 0}^\infty  {\left( {\begin{array}{*{20}{c}}
  {n + p + k - 1} \\
  k
\end{array}} \right)\frac{{{x^k}}}{{{{\left( {1 + x} \right)}^{n + p + k}}}}f\left( {\frac{k}{n}} \right),} \]
\item [(iv)]  When $\alpha  = 0,p = 0$ we obtain classical Baskakov operators~\cite{Baskakov:57} as:
 \[{L_n^{(\alpha )}(f;x)} = \sum\limits_{k = 0}^\infty  {\left( {\begin{array}{*{20}{c}}
  {n + k - 1} \\
  k
\end{array}} \right)\frac{{{x^k}}}{{{{\left( {1 + x} \right)}^{n + k}}}}f\left( {\frac{k}{n}} \right).} \]
\end{enumerate}
\end{enumerate}

\section{Preliminary Results}
In $1730$, J. Stirling~\cite{Stirling:30} introduced an important concept of numbers, useful in various branches of mathematics like number theory, calculus of Bernstein polynomials etc., known as Stirling numbers of first kind and afterwards Stirling numbers of second kind. Let $\R$ be the set of real numbers and $\N$, a collection of natural numbers with $\N_0= \N \cup \{0\}$. For $x \in \R$ and $i, j \in \N_0$, let $S(j, i)$ denote the Stirling numbers of second kind then
$${x^j} = \sum\limits_{i = 0}^j {S\left( {j,i} \right){{\left( x \right)}_i}},$$
with alternate form
\begin{equation}\label{x}
{x^j} = \sum\limits_{i = 0}^{j - 1} {S\left( {j,j - i} \right){{\left( x \right)}_{j - i}}},
\end{equation}
and has the following properties:
\begin{equation}\label{y}
S\left( {j,i} \right): = \left\{ \begin{gathered}
  1,\;\text{if}\;j = i = 0;j = i\;\text{or}\;j > 1,i = 1 \hfill \\
  0,\;\text{if}\;j > 0,i = 0 \hfill \\
  0,\;\text{if}\;j < i \hfill \\
  i.S\left( {j - 1,i} \right) + S\left( {j - 1,i - 1} \right),\;\text{if}\;j,i > 1. \hfill \\
\end{gathered}  \right.
\end{equation}
These Stirling numbers of second kind are very useful in calculating the moments of linear positive operators especially for higher order moments. In~\cite{Miclaus:12-1, Miclaus:12-2}, Dan obtained higher order moments for Bernstein type operators by using these numbers. In what follows, we shall find the moments of ${L_n^{(\alpha )}}$ given by ~\eqref{L1}with the help of same. Let us recall the monomials $e_j(x) = x^{j}, j \in \N_0$ be the test functions.
\begin{lem}\label{lem-1}
For the monomial \;$t^{j},$ where $j \in \N,$ and $t, x\in I$, we have
\[L_n^{(\alpha )}({t^{j}};x) = \frac{1}{{{n^j}}}\sum\limits_{i = 0}^{j - 1} {S\left( {j,j - i} \right)\phi _{n + p}^{j - i}\frac{{{x^{\left[ {j - i, - \alpha } \right]}}}}{{{1^{\left[ {j - i + \left( {\lambda  + 1} \right), - \alpha } \right]}}}}} ,\]
where
\[\phi _{n + p}^{j - i} = \left\{ \begin{gathered}
  {\left( {n + p} \right)_{\left( {j - i} \right)}},\;\lambda  =  - 1 \hfill \\
  {\left( {n + p} \right)^{\left( {j - i} \right)}},\;\lambda  = 0 \hfill \\
\end{gathered}  \right.,\]
${\left( y \right)_n}: = \prod\nolimits_{i = 0}^{n - 1} {\left( {y - i} \right)} ,{\left( y \right)_0}: = 1$ and ${\left( y \right)^n}: = \prod\nolimits_{i = 0}^{n - 1} {\left( {y + i} \right)} ,{\left( y \right)^{\left( 0 \right)}}: = 1$ are respectively the falling factorial and rising factorial with $y\in \R$ and $n\in \N.$
\end{lem}
\begin{proof}
Using the relation \eqref{x}, it is easy to derive above expression on the same lines as in \cite{Miclaus:12} (Page 54, Theorem 2.1.94.).
\end{proof}
\begin{lem}\label{lem-2}
For the Generalized positive linear operators \eqref{L1} hold
\begin{align*}\label{L2}
{L_n^{(\alpha )}(1;x)} &= 1,\;\;
L_n^{(\alpha )}(t;x) = \left( {\frac{{n + p}}{n}} \right)\frac{x}{{\left( {1 - \overline {\lambda  + 1} \alpha } \right)}},\\
L_n^{(\alpha )}({t^2};x)& = \left( {\frac{{n + p}}{{{n^2}}}} \right)\frac{1}{{\left( {1 - \lambda \alpha } \right)\left( {1 - \overline {\lambda  + 1} \alpha } \right)}}\\
& \times \left[ {\frac{{\left( {n + p + \lambda  + 1} \right)x(x + \alpha )}}{{1 - 2\overline {\lambda  + 1} \alpha }} + x\left( {1 + \lambda x} \right)} \right],
\end{align*}
\begin{align*}
  L_n^{(\alpha )}({t^3};x) &= \frac{{\left( {n + p} \right)x}}{{{n^3}\left( {1 - \overline {\lambda  + 1} \alpha } \right)}}\left[ {\frac{{\left( {n + p + 2\lambda  + 1} \right)\left( {n + p + 2\overline {2\lambda  + 1} } \right)(x + \alpha )(x + 2\alpha )}}{{\left( {1 - \overline {3\lambda  + 2} \alpha } \right)\left( {1 - \overline {5\lambda  + 3} \alpha } \right)}}} \right.\\
  &\hspace{.5cm}\left. { + \frac{{3\left( {n + p + 2\lambda  + 1} \right)(x + \alpha )}}{{\left( {1 - \overline {3\lambda  + 2} \alpha } \right)}} + 1} \right],
\end{align*}
and
\begin{align*}
  &L_n^{(\alpha )}({t^4};x) = \frac{{\left( {n + p} \right)x}}{{{n^4}\left( {1 - \overline {\lambda  + 1} \alpha } \right)}}\\
  &\hspace{.5cm}\times \left[ {\frac{{\left( {n + p + 2\lambda  + 1} \right)\left( {n + p + 2\overline {2\lambda  + 1} } \right)\left( {n + p + 3\overline {2\lambda  + 1} } \right)(x + \alpha )(x + 2\alpha )(x + 3\alpha )}}{{\left( {1 - \overline {3\lambda  + 2} \alpha } \right)\left( {1 - \overline {5\lambda  + 3} \alpha } \right)\left( {1 - \overline {7\lambda  + 4} \alpha } \right)}}} \right.\\
  &\hspace{.5cm}+ \frac{{6\left( {n + p + 2\lambda  + 1} \right)\left( {n + p + 2\overline {2\lambda  + 1} } \right)(x + \alpha )(x + 2\alpha )}}{{\left( {1 - \overline {3\lambda  + 2} \alpha } \right)\left( {1 - \overline {5\lambda  + 3} \alpha } \right)}}\\
  &\hspace{.5cm}\left. { + \frac{{7\left( {n + p + 2\lambda  + 1} \right)(x + \alpha )}}{{\left( {1 - \overline {3\lambda  + 2} \alpha } \right)}} + 1} \right].
\end{align*}
\end{lem}
\begin{proof}
From the definition of operators \eqref{L1}, we can obtain the moment for $j=0$,\\
 i.e.${L_n^{(\alpha )}(1;x)} = 1$

 Also by the application of  Lemma \ref{lem-1} for $j=1,2,3,4$ and taking into account the relation \eqref{y}, we can follow the values of remaining moments.
\end{proof}
Further, to obtain the central moments of generalized positive operators \eqref{L1}, we use the following result:
\begin{lem}\label{lem-3}\cite{GPR:06}
Let V be any linear operators then
\[V\left( {{{\left( {t - x} \right)}^j};x} \right) = V\left( {t^{j};x} \right) - \sum\limits_{i = 0}^{j - 1} {\left( {\begin{array}{*{20}{c}}
j\\
i
\end{array}} \right){x^{j - i}}} V\left( {{{\left( {t - x} \right)}^i};x} \right),\]
and in the case when $V\left( {t^{j};x} \right) = x^{j}$, for $j = 0,1$, then we get
\[V\left( {{{\left( {t - x} \right)}^3};x} \right) = V\left( {t^{3};x} \right) - {x^3} - 3xV\left( {{{\left( {t - x} \right)}^2};x} \right),\]
and
\[V\left( {{{\left( {t - x} \right)}^4};x} \right) = V\left( {t^{4};x} \right) - {x^4} - 4xV\left( {{{\left( {t - x} \right)}^3};x} \right) + 6{x^2}V\left( {{{\left( {t - x} \right)}^2};x} \right).\]
\end{lem}
\begin{lem}\label{lem-4}
The generalized linear positive operators \eqref{L1} satisfy
\begin{align}\label{z}
L_n^{(\alpha )}(t - x;x) &= \frac{{px + \overline {\lambda  + 1} \alpha }}{{n\left( {1 - \overline {\lambda  + 1} \alpha } \right)}},
\end{align}
\begin{align}\label{z1}
\nonumber &L_n^{(\alpha )}({\left( {t - x} \right)^2};x) = \frac{{n + p}}{{n\left( {1 - \lambda \alpha } \right)\left( {1 - \overline {\lambda  + 1} \alpha } \right)}}\left[ {\left( {1 - \lambda \alpha } \right)\left( {1 - \overline {\lambda  + 1} \alpha } \right)\frac{{n{x^2}}}{{n + p}}} \right.\\
  &\hspace{.5cm} + \frac{{\left( {n + p + \lambda  + 1} \right)x(x + \alpha )}}{{n\left( {1 - 2\overline {\lambda  + 1} \alpha } \right)}} + \frac{{x\left( {1 + \lambda x} \right)}}{n} - 2\left( {1 - \lambda \alpha } \right){x^2},
\end{align}
\begin{align}\label{z1}
&\nonumber L_n^{(\alpha )}\left( {{{\left( {t - x} \right)}^3};x} \right) = \frac{{\left( {n + p} \right)x(x + \alpha )}}{{{n^2}\left( {1 - \overline {\lambda  + 1} \alpha } \right)}}\left[ {\frac{{\left( {n + p + 2\lambda  + 1} \right)\left( {n + p + 2\overline {2\lambda  + 1} } \right)(x + 2\alpha )}}{{n\left( {1 - \overline {3\lambda  + 2} \alpha } \right)\left( {1 - \overline {5\lambda  + 3} \alpha } \right)}}} \right.\\
  \nonumber&\hspace{.5cm}+\left. {\frac{{3\left( {n + p + 2\lambda  + 1} \right)}}{{n\left( {1 - \overline {3\lambda  + 2} \alpha } \right)}} - \frac{{3\left( {n + p + \lambda  + 1} \right)x}}{{\left( {1 - \lambda \alpha } \right)\left( {1 - 2\overline {\left( {\lambda  + 1} \right)} \alpha } \right)}}} \right]\\
  \nonumber&\hspace{.5cm} + \frac{{\left( {n + p} \right)x}}{{n\left( {1 - \overline {\lambda  + 1} \alpha } \right)}}\left[ {\frac{1}{{{n^2}}}   - 3x\left( {\frac{{px + \overline {\lambda  + 1} \alpha }}{{\left( {n + p} \right)}}} \right) - \frac{3}{{\left( {1 - \lambda \alpha } \right)}}\frac{{x\left( {1 + \lambda x} \right)}}{n}} \right]\\
   &\hspace{.5cm}- 2{x^3}\left[ {2 - \frac{{3\left( {n + p} \right)}}{{n\left( {1 - \overline {\lambda  + 1} \alpha } \right)}}} \right],
\end{align}
\begin{align}\label{z2}
 &\nonumber L_n^{(\alpha )}\left( {{{\left( {t - x} \right)}^4};x} \right) = \frac{{\left( {n + p} \right)x(x + \alpha )}}{{{n^2}\left( {1 - \overline {\lambda  + 1} \alpha } \right)}}\\
 \nonumber&\hspace{.5cm}\times\left[ {\frac{{\left( {n + p + 2\lambda  + 1} \right)\left( {n + p + 2\overline {2\lambda  + 1} } \right)\left( {n + p + 3\overline {2\lambda  + 1} } \right)(x + 2\alpha )(x + 3\alpha )}}{{{n^2}\left( {1 - \overline {3\lambda  + 2} \alpha } \right)\left( {1 - \overline {5\lambda  + 3} \alpha } \right)\left( {1 - \overline {7\lambda  + 4} \alpha } \right)}}} \right.\\
 \nonumber&\hspace{.5cm} + \frac{{2\left( {3 - 2nx} \right)\left( {n + p + 2\lambda  + 1} \right)\left( {n + p + 2\overline {2\lambda  + 1} } \right)(x + 2\alpha )}}{{{n^2}\left( {1 - \overline {3\lambda  + 2} \alpha } \right)\left( {1 - \overline {5\lambda  + 3} \alpha } \right)}}\\
 \nonumber&\hspace{.5cm}\left. { + \frac{{\left( {7 - 12nx} \right)\left( {n + p + 2\lambda  + 1} \right)}}{{{n^2}\left( {1 - \overline {3\lambda  + 2} \alpha } \right)}} + \frac{{6\left( {n + p + \lambda  + 1} \right){x^2}}}{{\left( {1 - \lambda \alpha } \right)\left( {1 - 2\overline {\left( {\lambda  + 1} \right)} \alpha } \right)}}} \right]\\
 \nonumber&\hspace{.5cm} + \frac{{\left( {n + p} \right)x}}{{n\left( {1 - \overline {\lambda  + 1} \alpha } \right)}}\left[ {\frac{{1 - 4nx}}{{{n^3}}} + 8{x^2}\left( {\frac{{px + \overline {\lambda  + 1} \alpha }}{{\left( {n + p} \right)}}} \right)\frac{{6{x^2}\left( {1 + \lambda x} \right)}}{{n\left( {1 - \lambda \alpha } \right)}}} \right]\\
 &\hspace{.5cm}+ 3{x^4}\left[ {3 - \frac{{4\left( {n + p} \right)}}{{n\left( {1 - \overline {\lambda  + 1} \alpha } \right)}}} \right].
 \end{align}
\end{lem}
\begin{proof}
 The combined use of Lemma \ref{lem-2} and \ref{lem-3} will follow the proof.
\end{proof}
\begin{rem}\label{rem-1}
For sufficiently large $n$, Lemma \ref{lem-5} gives the following inequalities:
\begin{enumerate}
  \item [(i)] $L_n^{(\alpha )}\left( {{{\left( {t - x} \right)}^2};x} \right) \le {A_1}\frac{{\left( {1 + {x^2}} \right)}}{n},$
  \item [(ii)] $L_n^{(\alpha )}\left( {{{\left( {t - x} \right)}^4};x} \right) \le {A_2}\frac{{{{\left( {1 + {x^2}} \right)}^2}}}{n},$
\end{enumerate}
where $A_1$ and $A_2$ are some positive constants.
\end{rem}
\begin{lem}\label{lem-5}
For positive linear operators \eqref{L1}, there holds,
\[\left| {L_n^{\left( \alpha  \right)}\left( {f;x} \right)} \right| \leqslant\left\| f \right\|.\]
\end{lem}
\begin{proof}
From operators \eqref{L1}, we have
\begin{align*}
\left| {L_n^{\left( \alpha  \right)}\left( {f;x} \right)} \right| &= \left| {\sum\limits_{k = 0}^\infty  {w _{n,k}^{\left( \alpha  \right)}} (x)f\left( {\frac{k}{n}} \right)} \right| \leqslant \sum\limits_{k = 0}^\infty  {w _{n,k}^{\left( \alpha  \right)}} (x)\left| {f\left( {\frac{k}{n}} \right)} \right| \\
&\leqslant \sum\limits_{k = 0}^\infty  {w _{n,k}^{\left( \alpha  \right)}} (x)\sup \left| {f(x)} \right| = \left\| f \right\|.
\end{align*}
\end{proof}

\section{Direct Results}
Let ${C_B}\left( I\right)$ be the space of all the real valued continuous and bounded functions $f$ on the interval $I$, endowed with the norm
\[\left\| f \right\| = \mathop {\sup }\limits_{x \in I} \left| {f(x)} \right|.\]
For $f \in{C_B}\left(I \right)$, the Peetre's $K-$functional is defined by
\[{K_2}\left( {f;\delta } \right): = \mathop {\inf }\limits_{x \in C_B ^2\left( I\right)} \left\{ {\left\| {f - g} \right\| + \delta \left\| {g''} \right\|} \right\},\;\delta>0,\]
where $C_B^2\left( I\right) = \left\{ {g \in {C_B}\left( I\right):g',g'' \in C_B\left( I\right)} \right\}.$ By DeVore and Lorentz(~\cite{Devore:93}, p.177. Theorem 2.4) there exists an absolute constant $C>0$ such that
\begin{equation}\label{K1}
{K_2}\left( {f;\delta } \right) \leqslant C{\omega _2}\left( {f;\sqrt \delta  } \right),
\end{equation}
where ${\omega _2}\left( {f;\sqrt \delta  } \right)$ is second order modulus of continuity defined by
\[{\omega _2}\left( {f;\sqrt \delta  } \right) = \mathop {\sup }\limits_{0 < h < \sqrt \delta  } \mathop {\sup }\limits_{x \in I} \left| {f\left( {x + 2h} \right) - 2f(x + h) + f(x)} \right|.\]
Also the first order modulus of smoothness(or simply modulus of continuity) is given by
\[\omega \left( {f;\sqrt \delta  } \right) = \mathop {\sup }\limits_{0 < h \leqslant \sqrt \delta  } \mathop {\sup }\limits_{x \in I} \left| {f(x + h) - f(x)} \right|.\]

\begin{theorem}\label{thm-1}
For $f \in C_B (I)$, we have
\[\left| {L_n^{(\alpha )}\left( {f;x} \right) - f(x)} \right|\leqslant \omega \left( {f,\frac{{px + nx(\lambda  + 1)\alpha }}{{x\left( {1 - \left( {\lambda  + 1} \right)\alpha } \right)}}} \right) + C{\omega _2}\left( {f,\frac{{\sqrt {\psi _{n,\lambda }^{\left( \alpha  \right)}\left( x \right)} }}{2}} \right),\]
where $C$ is a positive constant and
\[\psi _{n,\lambda }^{\left( \alpha  \right)}(x) = L_n^{\left( \alpha  \right)}\left( {{{(t - x)}^2};x} \right) + {\left\{ {\frac{{px + nx\left( {\lambda  + 1} \right)\alpha }}{{n\left( {1 - \left( {\lambda  + 1} \right)\alpha } \right)}}} \right\}^2}.\]
\end{theorem}

\begin{proof}
First we consider auxiliary operators
\begin{equation}\label{L3}
{\hat L_n}^{\left( \alpha  \right)}\left( {f;x} \right) = L_n^{\left( \alpha  \right)}\left( {f;x} \right) + f(x) - f\left( {\frac{{n + p}}{n}.\frac{x}{{1 - \left( {\lambda  + 1} \right)\alpha}}} \right).
\end{equation}
For all $x \in I$ we observed that ${\hat L_n}^{\left( \alpha  \right)}\left( {f;x} \right)$ are linear such that
$${\hat L_n}^{\left( \alpha  \right)}\left( {1;x} \right)=1\;\;and\;\;{\hat L_n}^{\left( \alpha  \right)}\left( {t;x} \right)=x,$$
i. e., preserve linear functions. Therefore
\begin{equation}\label{L4}
{{\hat L}_n}^{\left( \alpha  \right)}\left( {t - x;x} \right) = 0.
\end{equation}
Let $g\in C_B^2 (I)$ and $t,x \in I$ then Taylor's theorem implies
\[g(t) = g(x) + (t - x)g'(x) + \int_x^t {(t - u)g''(u)du}, \]
we can write
\begin{align*}
\hat L_n^{\left( \alpha  \right)}\left( {g;x} \right) - g(x) &= g'(x)\hat L_n^{\left( \alpha  \right)}\left( {(t - x);x} \right) + \hat L_n^{\left( \alpha  \right)}\left( {\int_x^t {(t - u)g''(u)du} ;x} \right)\\
&=\hat L_n^{\left( \alpha  \right)}\left( {\int_x^t {(t - u)g''(u)du} ;x} \right)\\
&=L_n^{\left( \alpha  \right)}\left( {\int_x^t {(t - u)g''(u)du} ;x} \right)\\
&\hspace{.5cm}- \int_x^{\frac{{n + p}}{n}\left( {\frac{x}{{1 - \left( {\lambda  + 1} \right)\alpha}}} \right)} {\left( {\frac{{n + p}}{n} {\frac{x}{{1 - \left( {\lambda  + 1} \right)\alpha}}} - u} \right)g''(u)du}.
\end{align*}
Hence we have
\begin{align}\label{L5}
\nonumber\left| {\hat L_n^{\left( \alpha  \right)}\left( {g;x} \right) - g(x)} \right|  &\leqslant L_n^{\left( \alpha  \right)}\left( {\left| {\int_x^t {(t - u)g''(u)du} } \right|;x} \right)\\
&+ \left| {\int_x^{\frac{{n + p}}{n}\left( {\frac{x}{{1 - \left( {\lambda  + 1} \right)\alpha}}} \right)} {\left( {\frac{{n + p}}{n} {\frac{x}{{1 - \left( {\lambda  + 1} \right)\alpha}}} - u} \right)g''(u)du} } \right|.
\end{align}
Since $\left| {\int_x^t {(t - u)g''(u)du} } \right| \leqslant {(t - x)^2}\left\| {g''} \right\|$  and
\[\left| {\int_x^{\frac{{n + p}}{n}\frac{x}{{1 - \left( {\lambda  + 1} \right)\alpha }}} {\left( {\frac{{n + p}}{n}\frac{x}{{1 - \left( {\lambda  + 1} \right)\alpha }} - u} \right)} g''(u)du} \right|\leqslant{\left\{ {\frac{{n + p}}{n}\frac{x}{{1 - \left( {\lambda  + 1} \right)\alpha }} - x} \right\}^2}\|g''\|.\]
Therefore \eqref{L5} implies that
\begin{align}\label{L6}
\nonumber\left| {\hat L_n^{\left( \alpha  \right)}\left( {g;x} \right) - g(x)} \right| &\leqslant \left[ {L_n^{\left( \alpha  \right)}\left( {{{(t - x)}^2};x} \right) + {{\left\{ {\frac{{n + p}}{n}\left( {\frac{x}{{1 - \left( {\lambda  + 1} \right)\alpha }}} \right) - x} \right\}}^2}} \right]\left\| {g''} \right\|\\
&\nonumber\leqslant\left[ {L_n^{\left( \alpha  \right)}\left( {{{(t - x)}^2};x} \right) + {{\left\{ {\frac{{px + nx\left( {\lambda  + 1} \right)\alpha }}{{n\left( {1 - \left( {\lambda  + 1} \right)\alpha } \right)}}} \right\}}^2}} \right]\left\| {g''} \right\|\\
&= \psi _{n,\lambda }^{\left( \alpha  \right)}(x)\left\| {g''} \right\|,
\end{align}
Again using definition of auxiliary operators and from  Lemma \ref{lem-5}, we get
\begin{align*}
\left| {L_n^{\left( \alpha  \right)}\left( {f;x} \right) - f(x)} \right| &\leqslant \left| {\hat L_n^{\left( \alpha  \right)}\left( {(f - g);x} \right)} \right| + \left| {g(x) - f(x)} \right| + \left| {\hat L_n^{\left( \alpha  \right)}\left( {g;x} \right) - g(x)} \right|\\
&\hspace{.5cm}+ \left| {f\left( {\frac{{n + p}}{n}\frac{x}{{1 - (\lambda  + 1)\alpha }}} \right) - f(x)} \right|\\
&\leqslant 4\left\| {f - g} \right\| + \psi _{n,\lambda }^{(\alpha )}(x)\left\| g'' \right\| + \omega \left( {f;\frac{{px + nx\left( {\lambda  + 1} \right)\alpha }}{{n\left( {1 - \left( {\lambda  + 1} \right)\alpha } \right)}}} \right).
\end{align*}
Taking infimum on both the sides over $g \in C_B^2\left( I \right),$
\[\left| {L_n^{\left( \alpha  \right)}\left( {f;x} \right) - f(x)} \right| \leqslant 4{K_2}\left( {f;\frac{{\psi _{n,\lambda }^{(\alpha )}(x)}}{4}} \right) + \omega \left( {f;\frac{{px + nx\left( {\lambda  + 1} \right)\alpha }}{{n\left( {1 - \left( {\lambda  + 1} \right)\alpha } \right)}}} \right).\]
Hence by \eqref{K1}, we get
\[\left| {L_n^{\left( \alpha  \right)}\left( {f;x} \right) - f(x)} \right| \leqslant C{\omega _2}\left( {f;\frac{{\sqrt {\psi _{n,\lambda }^{(\alpha )}(x)} }}{2}} \right) + \omega \left( {f;\frac{{px + nx\left( {\lambda  + 1} \right)\alpha }}{{n\left( {1 - \left( {\lambda  + 1} \right)\alpha } \right)}}} \right).\]
\end{proof}
We consider the following Lipschitz-type space (see~\cite{OD:10})
\[Lip_M^*(\beta): = \left\{ {f \in {C_B}(I):\left| {f(y) - f(x)} \right| \leqslant M\frac{{{{\left| {y - x} \right|}^ \beta}}}{{{{(x + y)}^{ \beta/2}}}};x,y \in (0,\infty )} \right\},\]
where $M$ is a positive constant and $0 < \beta \leqslant 1$.
\begin{theorem}
For all $x\in I$ and $f \in Lip_M^*(\beta ),\;0<\beta  \in (0,1]$ we get
\begin{equation}\label{L7}
\left| {L_n^{(\alpha )}(f;x) - f(x)} \right| \leqslant M{\left( {\frac{{\phi _n^{(\alpha )}(x)}}{x}} \right)^{\beta /2}},
\end{equation}
where ${\phi _n^{(\alpha )}(x)} = L_n^{(\alpha )}\left( {{{(t - x)}^2};x} \right)$.
\end{theorem}
\begin{proof}
Assume that $\beta = 1$.  Then, for $f \in Li{p_M^*}(1)$, we have
\begin{align*}
\left| {L_n^{(\alpha )}(f;x) - f(x)} \right| &\leqslant \left| {\sum\limits_{k = 0}^\infty  {\omega _{n,k}^{(\alpha )}(x)f\left( {\frac{k}{n}} \right) - f(x)} } \right|\\
&\leqslant \sum\limits_{k = 0}^\infty  {\omega _{n,k}^{(\alpha )}(x)} \left| {f\left( {\frac{k}{n}} \right) - f(x)} \right|\leqslant M\sum\limits_{k = 0}^\infty  {\omega _{n,k}^{(\alpha )}(x)} \frac{{\left| {\frac{k}{n} - x} \right|}}{{{{\left( {\frac{k}{n} + x} \right)}^{1/2}}}}.
\end{align*}
Applying Cauchy-Schwarz inequality for sum and $\frac{1}{{\sqrt {\frac{k}{n} + x} }} \leqslant \frac{1}{{\sqrt x }}$, we have
\begin{align*}
\left| {L_n^{(\alpha )}(f;x) - f(x)} \right| &\leqslant \frac{M}{{\sqrt x }}\sum\limits_{k = 0}^\infty  {\omega _{n,k}^{(\alpha )}(x)} {\left\{ {{{\left( {\frac{k}{n} - x} \right)}^2}} \right\}^{1/2}}\\
&\leqslant \frac{M}{{\sqrt x }}{\left\{ {\sum\limits_{k = 0}^\infty  {\omega _{n,k}^{(\alpha )}(x)} } \right\}^{1/2}}{\left\{ {\sum\limits_{k = 0}^\infty  {\omega _{n,k}^{(\alpha )}} {{\left( {\frac{k}{n} - x} \right)}^2}} \right\}^{1/2}}\\
&\leqslant \frac{M}{{\sqrt x }}{\left\{ {L_n^{(\alpha )}(1;x)} \right\}^{1/2}}{\left\{ {L_n^{(\alpha )}\left( {{{(t - x)}^2};x} \right)} \right\}^{1/2}} = M{\left\{ {\frac{{\phi _n^{(\alpha )}(x)}}{x}} \right\}^{1/2}}.
\end{align*}
Therefore result is true for $\beta = 1$.

Now we prove the required result for $0<\beta<1$. Consider $f \in Lip_M^*(\beta )$
\begin{align*}
\left| {L_n^{(\alpha )}(f;x) - f(x)} \right| &\leqslant \sum\limits_{k = 0}^\infty  {\omega _{n,k}^{(\alpha )}(x)} \left| {f\left( {\frac{k}{n}} \right) - f(x)} \right|\leqslant M\sum\limits_{k = 0}^\infty  {\omega _{n,k}^{(\alpha )}(x)} \frac{{{{\left| {\frac{k}{n} - x} \right|}^\beta }}}{{{{\left( {\frac{k}{n} + x} \right)}^{\beta /2}}}}.
\end{align*}
 Using Holder's inequality for sum with $p = 2/\beta ,q = 2/(2 - \beta )$ and inequality $\frac{1}{{\sqrt {\frac{k}{n} + x} }} \leqslant \frac{1}{{\sqrt x }}$, we have
\begin{align*}
\left| {L_n^{(\alpha )}(f;x) - f(x)} \right| &\leqslant \frac{M}{{{x^{\beta /2}}}}\sum\limits_{k = 0}^\infty  {\omega _{n,k}^{(\alpha )}(x)} {\left\{ {{{\left( {\frac{k}{n} - x} \right)}^2}} \right\}^{\beta /2}}\\
&\leqslant \frac{M}{{{x^{\beta /2}}}}{\left\{ {\sum\limits_{k = 0}^\infty  {\omega _{n,k}^{(\alpha )}(x)} {{\left( {\frac{k}{n} - x} \right)}^2}} \right\}^{\beta /2}}{\left\{ {\sum\limits_{k = 0}^\infty  {\omega _{n,k}^{(\alpha )}(x)} } \right\}^{\frac{{2 - \beta }}{2}}}\\
&\leqslant M{\left\{ {\frac{{L_n^{(\alpha )}\left( {{{(t - x)}^2};x} \right)}}{x}} \right\}^{\beta /2}}=M{\left\{ {\frac{{\phi _n^{(\alpha )}(x)}}{x}} \right\}^{\beta /2}}.
\end{align*}
\end{proof}

\section{Weighted Approximation}

Gadjiev~\cite{Gadjiev:74, Gadjiev:76} studied the weight spaces ${C_\varphi }(I)$ and ${B_\varphi }(I)$ of real valued functions defined on $I$ with $\varphi(x)= 1+x^{2}$ and proved that Korovkin's theorem in general does not hold on these spaces.
Here
\[{B_\varphi }(I): = \left\{ {f:\left| {f(x)} \right| \leqslant {M_f}\varphi(x)} \right\},\]
with \[{\left\| f \right\|_\varphi } = \mathop {\sup }\limits_{x \in I} \frac{{\left| {f(x)} \right|}}{{\varphi (x)}},\]
and
\[{C_\varphi }(I): = \left\{ {f:f \in {B_\varphi }(I){\text{ and }}f{\text{ continuous}}} \right\},\] i.e. ${C_\varphi }(I) = C(I) \cap {B_\varphi }(I)$ is the subspace of ${B_\varphi }(I)$ containing continuous functions and
\[C_\varphi ^*(I): = \left\{ {f \in {C_\varphi }(I):\mathop {\lim }\limits_{x \to \infty } \frac{{\left| {f(x)} \right|}}{{\varphi (x)}} < \infty } \right\}.\]
But Korovkin's theorem holds in the space  $C_\varphi ^*(I)$.

The usual modulus of continuity of $f$ on $[0,b]$ is defined as:
\[{\omega _b}\left( {f,\delta } \right) = \mathop {\sup }\limits_{\left| {t - x} \right| \leqslant \delta } \mathop {\sup }\limits_{x,t \in \left[ {0,b} \right]} \left| {f(t) - f(x)} \right|.\]
\begin{theorem}\label{thm-2}
Let $f \in C_\varphi (I)$ then for operators $L_n^{\left( \alpha  \right)}\left( {f;x} \right)$ we have
\begin{equation}\label{L10}
{\left\| {L_n^{\left( \alpha  \right)}\left( {f;x} \right) - f(x)} \right\|_{C[0,b]}} \le 4{M_f}\left( {1 + {b^2}} \right)\phi _n^{\left( \alpha  \right)}(b) + 2{\omega _{b + 1}}\left( {f,\sqrt {\phi _n^{\left( \alpha  \right)}(b)} } \right),
\end{equation}
where $\phi _n^{\left( \alpha  \right)}(x) = L_n^{\left( \alpha  \right)}\left( {{{(t - x)}^2};x} \right)$.
\end{theorem}
\begin{proof}
Let $x \in \left[ {0, b} \right],\;t \in (b + 1,\infty ),$ and $t - x > 1$ we have
\begin{align}\label{f1}
\nonumber \left| {f(t) - f(x)} \right| &\leqslant {M_f}\varphi (t - x)={M_f}\left\{ {1 + {{(t - x)}^2}} \right\}={M_f}\left( {1 + {t^2} + {x^2} - 2xt} \right)\\
&\nonumber \leqslant{M_f}\left( {2 + {t^2} + {x^2}} \right)={M_f}\left\{ {2 + 2{x^2} + {{(t - x)}^2} + 2x(t - x)} \right\}\\
&\nonumber \leqslant {M_f}{(t - x)^2}\left\{ {3 + 2x + 2{x^2}} \right\} \leqslant 4{M_f}{(t - x)^2}\left( {1 + {x^2}} \right)\\
&\leqslant 4{M_f}{(t - x)^2}\left( {1 + {b^2}} \right).
\end{align}
If $x\in [0,b]$ and $t \in \left[ {0,b + 1} \right]$ then we have
\begin{equation}\label{f2}
\left| {f(t) - f(x)} \right| \leqslant {\omega _{b + 1}}\left( {\left| {t - x} \right|} \right) \leqslant \left( {1 + \frac{{\left| {t - x} \right|}}{\delta }} \right){\omega _{b + 1}}\left( {f,\delta } \right), \delta>0.
\end{equation}
Combining ~\eqref{f1} and~\eqref{f2}, we obtain
\[\left| {f(t) - f(x)} \right| \leqslant 4{M_f}\left( {1 + {b^2}} \right){(t - x)^2} + \left( {1 + \frac{{\left| {t - x} \right|}}{\delta }} \right){\omega _{b + 1}}\left( {f,\delta } \right).\]
Using Cauchy Schwartz inequality we get
\begin{align*}
\left| {L_n^{(\alpha )}(t) - f(x)} \right| &\leqslant 4{M_f}\left( {1 + {b^2}} \right)L_n^{(\alpha )}\left( {{{(t - x)}^2};x} \right) \\
&\hspace{.5cm}+ \left( {1 + \frac{1}{\delta }L_n^{(\alpha )}\left( {\left| {t - x} \right|;x} \right)} \right){\omega _{b + 1}}\left( {f,\delta } \right)\\
&\leqslant 4{M_f}\left( {1 + {b^2}} \right)L_n^{(\alpha )}\left( {{{(t - x)}^2};x} \right)\\
&\hspace{.5cm}+ \left[ {1 + \frac{1}{\delta }{{\left\{ {L_n^{(\alpha )}\left( {{{(t - x)}^2};x} \right)} \right\}}^{1/2}}} \right]{\omega _{b + 1}}\left( {f,\delta } \right)\\
&\leqslant 4{M_f}\left( {1 + {b^2}} \right)\phi _n^{\left( \alpha  \right)}(b) + 2\omega_{b + 1} \left( {f,\sqrt {\phi _n^{\left( \alpha  \right)}(b)} } \right).
\end{align*}
\end{proof}
\begin{theorem}\label{thm-3}
Let $f \in C_\varphi ^*(I)$ then, we have
\begin{equation}\label{L8}
\mathop {\lim }\limits_{n \to \infty } {\left\| {L_n^{\left( \alpha  \right)}\left( {f;x} \right) - f(x)} \right\|_\varphi } = 0.
\end{equation}
\end{theorem}
\begin{proof}
From~\cite{Gadjiev:76}, it's sufficient to verify the following three equations
\begin{equation}\label{L9}
\mathop {\lim }\limits_{n \to \infty } {\left\| {L_n^{\left( \alpha  \right)}\left( {t^{i};x} \right) - t^{i}} \right\|_\varphi } = 0,\;i = 0,1,2.
\end{equation}
Clearly equation~\eqref{L9} holds for $i=0$ as $L_n^{\left( \alpha  \right)}\left( {1;x} \right) = 1$. Now using Lemma \ref{lem-2} we have
\begin{align*}
{\left\| {L_n^{\left( \alpha  \right)}\left( {t;x} \right) - x} \right\|_\varphi }& = \mathop {\sup }\limits_{x \in I} \frac{1}{{\varphi (x)}}\left| {\frac{{n + p}}{n}\frac{x}{{\left( {1 - \overline {\lambda  + 1} \alpha } \right)}} - x} \right|\\
 &=\mathop {\sup }\limits_{x \in I} \frac{x}{{\varphi (x)}}\left| {\frac{{p + \overline {\lambda  + 1} \alpha n}}{{n\left( {1 - \overline {\lambda  + 1} \alpha } \right)}}} \right|\leqslant\frac{{p + \overline {\lambda  + 1} \alpha n}}{{n\left( {1 - \overline {\lambda  + 1} \alpha } \right)}},
\end{align*}
which implies that $\mathop {\lim }\limits_{n \to \infty } {\left\| {L_n^{\left( \alpha  \right)}\left( {{t};x} \right) - x} \right\|_\varphi }=0$.
Similarly we have
\begin{align*}
&{\left\| {L_n^{\left( \alpha  \right)}\left( {{t^2};x} \right) - {x^2}} \right\|_\varphi } = \mathop {\sup }\limits_{x \in I} \frac{1}{{\varphi (x)}}\left| {\frac{{n + p}}{{{n^2}}}\frac{x}{{\left( {1 - \lambda \alpha } \right)\left( {1 - \overline {\lambda  + 1} \alpha } \right)}}} \right.\\
&  \times \left. {\left[ {\frac{{\left( {n + p + \lambda  + 1} \right)x(x + \alpha )}}{{1 - 2\overline {\lambda  + 1} \alpha }} + x\left( {1 + \lambda x} \right)} \right] - {x^2}} \right|\\
&\leqslant \left| {\frac{{n + p}}{{{n^2}}}\frac{1}{{\left( {1 - \lambda \alpha } \right)\left( {1 - \overline {\lambda  + 1} \alpha } \right)}} \times \left[ {\frac{{\left( {n + p + \lambda  + 1} \right)(1 + \alpha )}}{{1 - 2\overline {\lambda  + 1} \alpha }} + \lambda  + 1} \right] - 1} \right|.
\end{align*}
Thus $\mathop {\lim }\limits_{n \to \infty } {\left\| {L_n^{\left( \alpha  \right)}\left( {{t^2};x} \right) - x^2} \right\|_\varphi }=0$.\\
This completes the proof.
\end{proof}
Finally we give the rate of convergence for the generalized positive linear operators~\eqref{L1} with the help of Y\"{u}ksel and Ispir~\cite{YI:06} works. For every $f \in C_\varphi ^*(I)$, the weighted modulus of continuity is defined by:
\[\Omega \left( {f;\delta } \right) = \mathop {\sup }\limits_{x \in I,0 < h \leqslant \delta } \frac{{\left| {f(x + h) - f(x)} \right|}}{{1 + {{\left( {x + h} \right)}^2}}}.\]
Properties of the weighted modulus of continuity are given by the following lemma:
\begin{lem}\label{lem-6}~\cite{YI:06}
Let $f \in C_\varphi ^*(I)$ then we have
\begin{enumerate}
  \item [(i)]   $\Omega \left( {f;\delta } \right)$ is a monotone increasing function of $\delta$;
  \item [(ii)]  $\mathop {\lim }\limits_{\delta  \to 0 + } \Omega \left( {f;\delta } \right) = 0$;
  \item [(iii)] for each $m\in \N,\;\Omega \left( {f;m\delta } \right) = m\Omega \left( {f;\delta } \right)$;
  \item [(iv)]  for each $\mu  \in I,\;\Omega \left( {f;\mu \delta } \right) = (1+\mu) \Omega \left( {f; \delta } \right)$.
\end{enumerate}
\end{lem}
\begin{theorem}\label{thm-4}
Let $f \in C_\varphi ^*(I)$ then for operators $L_n^{(\alpha )}(f;x)$ we have
\[\mathop {\sup }\limits_{x \in I} \frac{{\left| {L_n^{(\alpha )}(f;x) - f(x)} \right|}}{{{{\left( {1 + {x^2}} \right)}^{{5 \mathord{\left/
 {\vphantom {5 2}} \right.
 \kern-\nulldelimiterspace} 2}}}}} \le K\Omega \left( {f;\frac{1}{{\sqrt n }}} \right),\]
where $K$ is positive constant.
\end{theorem}
\begin{proof}
For ${x,t \in \left[ {0,\infty } \right)},\;\delta>0$ and by definition of $\Omega \left( {f;\delta } \right)$ and Lemma \ref{lem-4}, we have
\begin{align*}
\left| {f(t) - f(x)} \right| &\leqslant \left( {1 + {{\left( {x + \left| {t - x} \right|} \right)}^2}} \right)\Omega \left( {f;\left| {t - x} \right|} \right)\\
 &\leqslant \left( {1 + {{\left( {x + \left| {t - x} \right|} \right)}^2}} \right)\left( {1 + \frac{{\left| {t - x} \right|}}{\delta }} \right)\Omega \left( {f;\delta } \right)\\
 &\leqslant 2\left( {1 + {x^2}} \right)\left( {1 + {{(t - x)}^2}} \right)\left( {1 + \frac{{\left| {t - x} \right|}}{\delta }} \right)\Omega \left( {f;\delta } \right).
\end{align*}
Using the fact that $L_n^{(\alpha )}(f;x)$ are positive linear operators, we get
\begin{align}
\nonumber\left| {L_n^{(\alpha )}(f;x) - f(x)} \right| &\leqslant \sum\limits_{k = 0}^\infty  {w_{n,k}^{\left( \alpha  \right)}(x)} \left| {f\left( {\frac{k}{n}} \right) - f(x)} \right|\\
&\nonumber\leqslant 2\Omega \left( {f;\delta } \right)\left( {1 + {x^2}} \right)\sum\limits_{k = 0}^\infty  {w_{n,k}^{\left( \alpha  \right)}(x)\left\{ {1 + {{\left( {\frac{k}{n} - x} \right)}^2}} \right\}\left( {1 + \frac{{\left| {\frac{k}{n} - x} \right|}}{\delta }} \right)} \\
&\nonumber\leqslant 2\Omega \left( {f;\delta } \right)\left( {1 + {x^2}} \right)\left[ {\sum\limits_{k = 0}^\infty  {w_{n,k}^{\left( \alpha  \right)}(x)}  + } \right.\sum\limits_{k = 0}^\infty  {w_{n,k}^{\left( \alpha  \right)}(x)} {\left( {\frac{k}{n} - x} \right)^2}\\
&\nonumber\hspace{.5cm}\left. { + \sum\limits_{k = 0}^\infty  {w_{n,k}^{\left( \alpha  \right)}(x)} \frac{{\left| {\frac{k}{n} - x} \right|}}{\delta } + \sum\limits_{k = 0}^\infty  {w_{n,k}^{\left( \alpha  \right)}(x)} \left( {\frac{k}{n} - x} \right)^2\frac{{\left| {\frac{k}{n} - x} \right|}}{\delta }} \right]\\
&\nonumber\leqslant 2\Omega \left( {f;\delta } \right)\left( {1 + {x^2}} \right)\Bigg[ {1 + L_n^{\left( \alpha  \right)}\left( {{{(t - x)}^2};x} \right)}\\
&\hspace{.5cm} { + {L_n^{\left( \alpha  \right)}\left( {\left( {1 + {{(t - x)}^2}} \right)\frac{{\left| {t - x} \right|}}{\delta };x} \right)} } \Bigg].
\end{align}
Applying Cauchy Schwarz inequality and in view of Remark \ref{rem-1}, we obtain
\begin{align*}
 &\left| {L_n^{(\alpha )}(f;x) - f(x)} \right| \le 2\Omega \left( {f;\delta } \right)\left( {1 + {x^2}} \right)\left[ {1 + L_n^{\left( \alpha  \right)}\left( {{{(t - x)}^2};x} \right)} \right.\\
 &\hspace{.5cm}\left. { + \frac{1}{\delta }\left\{ {{{\left( {L_n^{\left( \alpha  \right)}\left( {{{(t - x)}^2};x} \right)} \right)}^{{1 \mathord{\left/
 {\vphantom {1 2}} \right.
 \kern-\nulldelimiterspace} 2}}} + {{\left( {L_n^{\left( \alpha  \right)}\left( {{{(t - x)}^2};x} \right)} \right)}^{{1 \mathord{\left/
 {\vphantom {1 2}} \right.
 \kern-\nulldelimiterspace} 2}}}{{\left( {L_n^{\left( \alpha  \right)}\left( {{{(t - x)}^4};x} \right)} \right)}^{{1 \mathord{\left/
 {\vphantom {1 4}} \right.
 \kern-\nulldelimiterspace} 4}}}} \right\}} \right]\\
 &\hspace{.5cm}\le K{\left( {1 + {x^2}} \right)^{{5 \mathord{\left/
 {\vphantom {5 2}} \right.
 \kern-\nulldelimiterspace} 2}}}\Omega \left( {f;\frac{1}{{\sqrt n }}} \right),
\end{align*}
where
$\delta = \frac{1}{{\sqrt n }}$ and $K = 2\left( {1 + {A_1} + \sqrt {{A_1}}  + \sqrt {{A_1}{A_2}} } \right).$
Hence the result.
\end{proof}

\end{document}